\documentclass[a4paper, 12pt,reqno]{amsart}

\usepackage[margin=1in]{geometry, stmaryrd}
\usepackage{dirtytalk}
\usepackage{bigints}
\usepackage{amssymb,latexsym}
\usepackage{graphicx}
\usepackage{float}
\graphicspath{ {./images/} }
\usepackage{mathtools}
\usepackage{color}
\usepackage[hyphens]{url}
\usepackage[T1]{fontenc}
\usepackage{hyperref}
\usepackage{amsmath}
\usepackage{mathdots}
\usepackage{breqn}
\usepackage[toc,page]{appendix}
\usepackage{url}    
\usepackage{breqn}
\usepackage{breakurl}
\usepackage{comment}
\usepackage{thmtools}
\usepackage{thm-restate}
\usepackage{mathrsfs}
\usepackage{fancyhdr} 
\usepackage{placeins}
\usepackage{leftidx}
\usepackage{bbm, dsfont}
\allowdisplaybreaks
\makeatletter
\def\imod#1{\allowbreak\mkern10mu({\operator@font mod}\,\,#1)}
\makeatother
\theoremstyle{plain}
\numberwithin{equation}{section}
\newtheorem{thm}{Theorem}[section]
\newtheorem*{theorem*}{Theorem}
\newtheorem{theorem}[thm]{Theorem}
\newtheorem{lemma}[thm]{Lemma}
\newtheorem{corollary}[thm]{Corollary}

\newtheorem{definition}[thm]{Definition}

\newtheorem{proposition}[thm]{Proposition}

\newtheorem{remark}[thm]{Remark}

\newcommand{\bigslant}[2]{{\raisebox{.2em}{$#1$}\left/\raisebox{-.2em}{$#2$}\right.}}
\newcommand{\sfrac}{\genfrac{}{}{}1}

\DeclareMathOperator{\GL}{GL}

\DeclareMathOperator{\Su}{Supp}
\DeclareMathOperator{\Hom}{Hom}
\DeclareMathOperator{\N}{N}
\DeclareMathOperator{\I}{I}
\DeclareMathOperator{\E}{E}
\DeclareMathOperator{\F}{F}

\DeclareMathOperator{\K}{K}

\DeclareMathOperator{\D}{D}
\DeclareMathOperator{\J}{J}

\DeclareMathOperator{\T}{T}

\DeclareMathOperator{\Mat}{Mat}
\DeclareMathOperator{\Ind}{Ind}
\DeclareMathOperator{\ind}{ind}

\title{D\lowercase{istinguished} S\lowercase{imple} S\lowercase{upercuspidal} R\lowercase{epresentations} \lowercase{of} $p$-\lowercase{adic} $\GL(n)$}

\author{D\lowercase{avid} C.~L\lowercase{uo}}

\address{School of Mathematics, University of Minnesota, Minneapolis, MN 55455, United States} \email{luo00275@umn.edu}


\begin{document}

\begin{abstract}
    Let $\E/\F$ be an unramified quadratic extension of non-archimedean local fields with odd residual characteristic.~In this paper, we give equivalent conditions for a simple supercuspidal representation $\pi$ of $\GL(n, \E)$ to be distinguished by $\GL(n, \F)$ in terms of its defining maximal simple type and twisted gamma factors.~Furthermore, we prove that the collection of twisted gamma factors evaluated at $\sfrac{1}{2}$ between $\pi$ and all unitary, tamely ramified quasi-characters of $\E^{\times}$ that are trivial on $\F^{\times}$ is sufficient to determine whether $\pi$ is distinguished by $\GL(n, \F)$.
\end{abstract}

\maketitle


\section{Introduction}

Let $\E/\F$ be an unramified quadratic extension of non-archimedean local fields with odd residual characteristic $p$.~Given a irreducible smooth complex representation $\tau$ of the general linear group $\GL(n,\E)$, one may ask whether $\tau$ admits a non-zero $\GL(n,\F)$-invariant linear form.~More precisely, one asks whether the space
\[
\Hom_{\GL(n, \, \F)}(\tau,\mathbbm{1})
\]
is non-zero, where $\mathbbm{1}$ denotes the trivial quasi-character of $\GL(n, \F)$.~If $\tau$ possesses such a linear form, we say that it is \textit{distinguished} by $\GL(n,\F)$.

The study of this question can be traced back to the 1980s.~The works 
\cite{Harder-Langlands-Rapoport, Jacquet-Ye} describe the initial motivation for identifying distinguished representations in the global setting, while the papers \cite{Flicker, Hakim} discuss applications in the non-archimedean local setting involving functoriality transfers from unitary groups \cite{Secherre}.~In 2019, S\'echerre gave explicit criteria for a complex \textit{$\sigma$-self-dual} supercuspidal representation $\pi$ (and likewise for $\sigma$-self-dual $\ell$-modular cuspidal representations with $\ell \neq p$) to be distinguished by $\GL(n,\F)$, in terms of invariants attached to its associated maximal simple type \cite[Theorem~10.9]{Secherre}.

Restricting to complex supercuspidal representations, one can also detect distinction by $\GL(n, \F)$ using \textit{twisted gamma factors}.~Let $\psi_{\E}$ be a non-trivial additive quasi-character of $\E$ that is trivial on $\F$.~Then to any pair of supercuspidal representations $\pi$ and $\tau$ of $\GL(n,\E)$, we can attach a corresponding twisted gamma factor $\gamma(s, \, \pi \times \tau, \, \psi_{\E})$, which may be defined via the Rankin--Selberg convolution integrals \cite{JPSS} or via the Langlands--Shahidi method \cite{Shahidi}.~The history of this approach is summarized in \cite[pp.~202--204]{Aubert}.

In 2015, Hakim and Offen gave an equivalent condition for supercuspidal representations of $\GL(n, \E)$ to be distinguished by $\GL(n, \F)$ in terms of these gamma factors \cite[Theorem 1.5.1]{Hakim-Offen}:
\begin{theorem}[Hakim--Offen, 2015]\label{n-2}
    Let $\pi$ be a supercuspidal representation of $\GL(n, \E)$ ($n \geq 3$) with central character trivial on $\F^{\times}$.~Then $\pi$ is distinguished by $\GL(n, \F)$ if and only if 
    \[
    \gamma\left(\sfrac{1}{2},\, \pi \times \tau, \, \psi_{\E}\right) = 1
    \] 
    for all irreducible generic representations $\tau$ of $\GL(r,\E)$ that are distinguished by $\GL(r,\F)$, where $1 \leq r \leq n-2$.
\end{theorem}
\noindent We note that a similar result for supercuspidal representations of $\GL(2, \E)$ also holds \cite[Theorem 4.1]{Hakim}.~Mimicking the local converse theorem, Nien conjectured that the upper bound of $n-2$ on $r$ in Theorem \ref{n-2} can be lowered to $\left\lfloor \sfrac{n}{2} \right\rfloor$ \cite[Conjecture 5.6, p.~203]{Aubert}.

Although progress on Nien's conjecture has been made for depth-zero supercuspidal representations \cite{Kurinczuk, Nien}, the case of positive-depth supercuspidal representations has not yet been studied.~In this paper, we focus on \textit{simple supercuspidal representations} (also referred to as epipelagic supercuspidal representations) of $\GL(n, \E)$, i.e., those supercuspidals of minimal positive depth $\frac{1}{n}$.~Letting $k_{\E}$ denote the residue field of $\E$, we parametrize isomorphism classes of simple supercuspidals by triples $\left(\overline{v},  \, \varphi,  \, \zeta \right)$ where $\overline{v} \in k_{\E}^{\times}$, $\varphi$ is a quasi-character of $k_{\E}^{\times}$, and $\zeta \in \mathbb{C}^{\times}$.

The main result we prove is the following. 
\begin{theorem}\label{TheoremMain}
Let $\pi:= \pi_{(\overline{v},\,\varphi,\,\zeta)}$ be a simple supercuspidal representation of $\GL(n, \E)$ with central character trivial on $\F^{\times}$.~The following statements are equivalent:
\begin{enumerate}
    \item $\pi$ is distinguished by $\GL(n,\F)$.
    \item For all unitary, tamely ramified quasi-characters $\lambda$ 
    of $\E^{\times}$ that are trivial on $\F^{\times}$,
    \[
    \gamma\left(\sfrac{1}{2},\,
    \pi \times \lambda,\,
    \psi_{\E} \right) = 1.
    \]
    \item The triple $(\overline{v}, \, \varphi, \, \zeta)$ satisfies: 
    $\overline{v} \in k_{\F}^{\times}$, 
    $\varphi$ is trivial on $k_{\F}^{\times}$, and $\zeta = 1$.
\end{enumerate}
\end{theorem}

Our initial decision to restrict to tamely ramified (depth-zero) quasi-characters of $\E^{\times}$ is motivated by the fact that the twisted gamma factors attached to a simple supercuspidal representation $\pi$ and a quasi-character $\lambda$ of depth strictly greater than $\frac{2}{n}$ reduce to the determinant form described in \cite[Equation (3), p.~90]{Deligne}.~The only situation in which positive-depth quasi-characters of $\E^{\times}$ do not yield this form occurs when $\pi$ is a simple supercuspidal representation of $\GL(2,\E)$ and $\lambda$ has depth-one.~In Appendix \ref{Appendix}, we compute this gamma factor and show that it does not provide additional information beyond that obtained from tamely ramified twists.

Theorem \ref{TheoremMain} reduces the upper bound on $r$ from $n-2$ (in Theorem \ref{n-2}) to $1$ for simple supercuspidal representations and supports Nien's conjecture.~Furthermore, it mimics the local converse theorem as the collection of twisted gamma factors between a simple supercuspidal representation $\pi$ and all tamely ramified quasi-characters determines $\pi$ up to isomorphism \cite[Remark 3.18]{Adrian-Liu}, \cite[Theorem~1.1]{Xu}.~From Theorem \ref{TheoremMain}, it is natural to ask whether the same collection of twisted gamma factors used to classify an arbitrary supercuspidal representation of $\GL(n, \E)$ up to isomorphism is sufficient to detect distinction by $\GL(n, \F)$.

Let us now outline the structure of this paper.~In Section \ref{Prelim} we give the necessary background and preliminary information needed for the remainder of this article.~Specifically, in Subsection \ref{SimpleSupercuspidalRepresentations}, we recall the structure of simple supercuspidal representations via maximal simple types.~Furthermore, we define twisted gamma factors using Rankin--Selberg convolution integrals in Subsection \ref{TwistedLocalFactors}.~Next, in Subsection \ref{Distinction}, we give necessary conditions for a irreducible smooth complex representation to be distinguished by $\GL(n, \F)$.~Finally, in Section \ref{Proof}, we present two auxiliary lemmas and prove Theorem \ref{TheoremMain}.

\section*{Acknowledgments}

The author is grateful to Robert Kurinczuk for suggesting this problem and to Sarbartha Bhattacharya, Dihua Jiang, Sagnik Mukherjee, and Shaun Stevens for helpful discussions.

\section{Preliminaries and Background}\label{Prelim}

\subsection{Notation}

Let $\E/\F$ be an unramified quadratic extension of non-archimedean local fields with odd residual characteristic $p$.~Associated with $\E$, let $\text{val}_{\E}$ denote the normalized discrete valuation, $\mathcal{O}_{\E}$ the valuation ring, $\mathcal{P}_{\E}$ the unique maximal ideal, $k_{\E}$ the residue field with cardinality $q_{\E}$, and $\mu_{\E}'$ the group of roots of unity with order relatively prime to $p$.~Let $\varpi_{\E}$ be a fixed uniformizer of $\E$ and let $\psi_{\E}$ be an additive quasi-character of $\E$ with conductor $\mathcal{P}_{\E}$; that is, $\psi_{\E}$ is non-trivial on $\mathcal{O}_{\E}$ but trivial on $\mathcal{P}_{\E}$.~We also impose the condition that $\psi_{\E}$ is trivial on $\F$.~For an element $x \in \mathcal{O}_{\E}^{\times}$, we denote its reduction in $k_{\E}^{\times}$ by $\overline{x}$.~We will use analogous notations for $\F$.

Let $\N(n, \E)$ denote the unipotent radical of the standard Borel subgroup of invertible upper triangular matrices of $\GL(n, \E)$, and $\text{P}(n, \E)$ the standard mirabolic subgroup of $\GL(n, \E)$.~Furthermore, let $\Mat(n \times m, \E)$ denote the space of $n \times m$ matrices with coefficients in $\E$ and $I_{n}$ the $n \times n$ identity matrix of $\GL(n, \E)$.~Next, we let $\psi_{n}$ denote the standard smooth non-degenerate quasi-character of $\N(n, \E)$ defined by
\[
\psi_{n}(u) = \psi_{\E}\left(\sum_{i=1}^{n-1}u_{i, i+1} \right)
\]
where $u = (u_{i, j}) \in \N(n, \E)$.

Let $dh$ be a multiplicative Haar measure on $\E^{\times}$ and $dx$ an additive Haar measure on $\E$ such that 
\[
dh = \frac{dx}{|x|} \; \; \text { and } \; \; \int_{\mathcal{O}_{\E}}dx = 1.
\]
Lastly, let $\Ind$ and $\ind$ denote smooth and compact induction, respectively.

\subsection{Simple supercuspidal representations}\label{SimpleSupercuspidalRepresentations}

In this subsection, we briefly recall the construction of \textit{simple supercuspidal representations} of $\GL(n, \E)$, $n \geq 2$, via the theory of maximal simple types (\cite{Knightly} and \cite[Example 2.18]{Ye}).~Let $\mathfrak{I}_{n}$ be the standard minimal hereditary $\mathcal{O}_{\E}$-order in $\mathrm{Mat}(n \times n, \E)$; that is, the space of $n \times n$ matrices with coefficients in $\mathcal{O}_{\E}$ whose image modulo $\mathcal{P}_{\E}$ is an upper triangular matrix with entries in $k_{\E}$.~Furthermore, let $\mathfrak{P}_{n}$ denote the Jacobson radical of $\mathfrak{I}_{n}$; that is, the space of $n \times n$ matrices with coefficients in $\mathcal{O}_{\E}$ whose image modulo $\mathcal{P}_{\E}$ is a strictly upper triangular matrix with entries in $k_{\E}$.

Next, let $v \in \mathcal{O}_{\E}^{\times}$ and consider the simple stratum $\left[\mathfrak{I}_{n}, 1, 0, \beta_{v}\right]$ where
\[
\beta_{v} = \begin{pmatrix}
    & & & & (v\varpi_{\E})^{-1} \\
    1 & & & & \\
     & 1 & & & \\
      & & \ddots & & \\
      & & & 1 &
\end{pmatrix}
\]
is minimal over $\E$, and $\E_{v} = \E \left[ \beta_{v} \right]$ is a degree $n$ totally ramified extension of $\E$.~Associated with the simple stratum above are the open normal subgroups of $\mathfrak{I}_{n}^{\times}$:
\[
\text{H}^{1}\left(\beta_{v}, \mathfrak{I}_{n}\right) = \J^{1}\left(\beta_{v}, \mathfrak{I}_{n}\right) = U^{1}\left(\mathfrak{I}_{n}\right) = \I_{n} + \mathfrak{P}_{n}
\]
and the open compact subgroup $\J_{v} := \J\left(\beta_{v}, \mathfrak{I}_{n}\right) = \mathcal{O}_{\E}^{\times} \, U^{1}\left( \mathfrak{I}_{n}\right)$.

Let $\psi_{\beta_{v}}$ be a quasi-character of $U^{1}\left(\mathfrak{I}_{n}\right)$ defined by
\[
\psi_{\beta_{v}}(x) = \left(\psi_{\E} \circ \text{Tr}_{\Mat(n \times n, \, \E)/\E}\right)\left(\beta_{v}(x-1)\right)
\]
where $\text{Tr}_{\Mat(n \times n, \E)/\E}$ denotes matrix trace.~Furthermore, let $\varphi$ be a quasi-character of $k_{\E}^{\times}$.~From these quasi-characters, we may form a $\beta_{v}$-extension 
\[
\kappa_{(v,  \, \varphi)} \, \colon \, \mathcal{O}_{\E}^{\times}  \, U^{1}\left(\mathfrak{I}_{n}\right) \to \mathbb{C}^{\times}
\]
by letting
\[
\kappa_{(v,  \, \varphi)}(xy) := \varphi(x)\psi_{\beta_{v}}(y)
\]
where $x \in \mathcal{O}_{\E}^{\times}$ and $y \in U^{1}\left(\mathfrak{I}_{n}\right)$ with $\varphi$ inflated as a quasi-character of $\mathcal{O}_{\E}^{\times}$ that is trivial on $1 + \mathcal{P}_{\E}$.

To form a maximal simple type with the data above, we take the pair $\left(\J_{v},  \kappa_{(v,  \, \varphi)}\right)$.~Moreover, to form an \textit{extended maximal simple type}, let $\textbf{J}_{v} = \E_{v}^{\times}\J_{v}$.~Next, we extend $\kappa_{(v,  \, \varphi)}$ to a quasi-character $\Lambda_{(v,  \, \varphi,  \, \zeta)}$ on $\textbf{J}_{v}$ by setting $\Lambda_{(v,  \, \varphi,  \, \zeta)}\left(\beta_{v}\right) = \zeta \in \mathbb{C}^{\times}$ so that 
\[
\Lambda_{(v,  \, \varphi,  \, \zeta)}(\varpi_{\E}) = \left(\zeta^{n}  \varphi(v)\right)^{-1}.
\]
We call the pair $\left(\textbf{J}_{v}, \Lambda_{(v,  \, \varphi,  \, \zeta)}\right)$ an extended maximal simple type.~Hence, we now define a simple supercuspidal representation:
\begin{definition}
    \normalfont 
    Let $\left(\textbf{J}_{v},  \Lambda_{(v,  \, \varphi,  \, \zeta)}\right)$ be an extended maximal simple type.~A \textit{simple supercuspidal representation} of $\GL(n, \E)$ is a supercuspidal representation of the form 
    \[
    \pi_{(v,  \, \varphi,  \, \zeta)} = \ind_{\textbf{J}_{v}}^{\GL(n, \E)}\Lambda_{(v,  \, \varphi,  \, \zeta)}.
    \]
\end{definition}
Simple supercuspidal representations are those supercuspidals of $\GL(n, \E)$ with minimal positive depth $\frac{1}{n}$.~Let $\omega_{(v,  \, \varphi,  \, \zeta)}$ denote the central character of $\pi_{(v,  \, \varphi,  \, \zeta)}$, then 
\[
\omega_{(v,  \, \varphi,  \, \zeta)} = \varphi \, \text{ on } \mathcal{O}_{\E}^{\times}
\]
and 
\[
\omega_{(v,  \, \varphi,  \, \zeta)}\left(\varpi_{\E}\right) = \Lambda_{(v,  \, \varphi,  \, \zeta)}\left(\varpi_{\E}\right).
\]

In Section \ref{Proof}, it will be convenient to decompose $\Lambda_{(v,  \, \varphi,  \, \zeta)}$ in the manner of \cite[Subsection 5.2]{Secherre} as
\begin{equation}\label{LambdaDecomposition}
\Lambda_{(v,  \, \varphi,  \, \zeta)} = \boldsymbol\kappa \otimes \boldsymbol\rho
\end{equation}
where $\boldsymbol\kappa$ is a quasi-character of $\textbf{J}_{v}$ extending $\psi_{\beta_{v}}$, and $\boldsymbol\rho$ is a quasi-character of $\textbf{J}_{v}$ trivial on $\J^{1}$.~We note that the decomposition of $\Lambda_{(v, \, \varphi, \, \zeta)}$ in (\ref{LambdaDecomposition}) is not unique in general; however, it suffices to fix one such decomposition where $\boldsymbol\kappa$ is distinguished in the sense of \cite{Secherre} to apply the results of \cite[Sections 9 and 10]{Secherre}.

Lastly, we parametrize simple supercuspidal representations here via $(v,  \, \varphi,  \, \zeta)$ as there exists a bijection between the set of isomorphism classes $\mathcal{A}_{\text{simple}}^{(n)}$ of simple supercuspidal representations of $\GL(n, \E)$ and the set of triples
    \[
     \T_{\text{simple}} := \left\{\left(\overline{v}, \, \varphi, \, \zeta\right) \, : \, \overline{v} \in k_{\E}^{\times}, \, \varphi \text{ is a quasi-character of } k_{\E}^{\times}, \, \zeta \in \mathbb{C}^{\times}  \right\}
    \] \cite[Proposition 1.3]{Imai}:
\begin{proposition}[Imai--Tsushima, 2018]\label{Bijection}
    The following is a bijection:
        \begin{align*}
            \mathcal{A}_{\text{simple}}^{(n)} &\longleftrightarrow \T_{\text{simple}} \\
            \left[\pi_{(v,  \, \varphi,  \, \zeta)}\right] & \longleftrightarrow \left(\overline{v}, \, \varphi, \, \zeta\right).
        \end{align*}
\end{proposition}

\subsection{Twisted gamma factors}\label{TwistedLocalFactors}
In this subsection, we define \textit{twisted gamma factors} via the Rankin--Selberg convolution integrals.~First, we recall the following (\cite[Section 2]{Adrian-Liu} and \cite{BH-2}): an irreducible admissible representation $\left(\pi, V_{\pi}\right)$ of $\GL(n, \E)$ is called \textit{generic} if 
\[
\Hom_{\GL(n, \, \E)}\left(\pi,  \Ind_{\N(n, \, \E)}^{\GL(n, \, \E)}\psi_{n}\right) \neq 0.
\]
By the uniqueness of local Whittaker models, this $\Hom$-space is at most 1-dimensional. From Frobenius reciprocity,
\[
\Hom_{\GL(n, \, \E)}\left(\pi,  \Ind_{\N(n, \, \E)}^{\GL(n, \, \E)}\psi_{n}\right) \cong \Hom_{\N(n, \, \E)}\left(\pi\big|_{\N(n, \, \E)},  \psi_{n}\right).
\]
Therefore, $\Hom_{\N(n, \, \E)}\left(\pi \big|_{\N(n, \, \E)},  \psi_{n}\right)$ is also at most 1-dimensional.~Assume that $\left(\pi, V_{\pi}\right)$ is generic.~Fix a non-zero functional $l \in \Hom_{\GL(n, \, \E)}\left(\pi,  \Ind_{\N(n, \, \E)}^{\GL(n, \, \E)}\psi_{n}\right)$, which is unique up to scalar.~The \textit{Whittaker function} attached to a vector $v \in V_{\pi}$ is defined by
\[
W_{v}(g) := l \left(\pi(g)v\right), \, \text{ for all } g \in \GL(n, \E),
\]
so that $W_{v} \in \Ind_{\N(n, \, \E)}^{\GL(n, \, \E)}\psi_{n}$.~The space
\[
W(\pi, \psi_{n}) := \{ W_{v} : v \in V_{\pi}\}
\]
is called the \textit{Whittaker model} of $\pi$ and $\GL(n, \E)$ acts on it by right translation.~It is easy to see that the Whittaker model of $\pi$ is independent of the choice of the non-zero functional $l$.

Suppose that $\pi$ is a generic representation of $\GL(n, \E)$ with central character $\omega_{\pi}$ and $\chi$ is a quasi-character of $\E^{\times}$.~For $g \in \GL(n, \E)$, let $\leftidx^{t}g$ denote the matrix transpose of $g$.~Furthermore, let 
\[
w_{n, 1} = \begin{pmatrix}
    1 & \\
    & w_{n-1}
\end{pmatrix} \in \GL(n, \E), \text{ where \, } w_{n-1} = \begin{pmatrix}
    & & 1 \\
    & \iddots & \\
    1 & &
\end{pmatrix} \in \GL(n-1, \E).
\]

Next, let $W_{\pi} \in W\left(\pi, \psi_{n}\right)$.~We define the \textit{Rankin--Selberg convolution integrals} $\widetilde{\Psi}$ and $\Psi$ attached to $\pi$ and $\chi$ by 
\begin{align*}
    & \widetilde{\Psi}\left(s; W_{\pi}, \chi\right) \\
    &= \int\displaylimits_{\E^{\times}}\int\displaylimits_{\Mat(n-2 \times 1, \, \E)}W_{\pi}\left(\begin{pmatrix}
        h & & \\
        x & I_{n-2} & \\
        & & 1
    \end{pmatrix}\right)\chi(h)|\det(h)|^{s-\sfrac{n-1}{2}} \, dx \, dh,
\end{align*}
and 
\[
\Psi\left(s; W_{\pi}, \chi\right) = \int\displaylimits_{\E^{\times}} W_{\pi}\left(\begin{pmatrix}
    h & \\
    & I_{n-1}
\end{pmatrix}\right)\chi(h)|\det(h)|^{s-\sfrac{n-1}{2}} \, dh.
\]
These integrals are absolutely convergent for $\text{Re}(s)$ sufficiently
large and are rational functions of $q_{\E}^{-s}$ \cite[Section 2.7]{JPSS}.

From these integrals, we obtain the following equation which gives the twisted gamma factor $\gamma(s, \pi \times \chi, \psi_{\E})$:
\begin{thm}\cite[Section 2.7]{JPSS}\label{GammaFactor2.7}
    There is a rational function $\gamma(s, \pi \times \chi, \psi_{\E}) \in \mathbb{C}\left(q_{\E}^{-s}\right)$ such that
    \[
    \widetilde{\Psi}\left(1-s; \rho(w_{n, 1})\widetilde{W}_{\pi}, \chi^{-1}\right) = \chi(-1)^{n-1}\gamma(s, \pi \times \chi, \psi_{\E})\Psi\left(s; W_{\pi}, \chi\right),
    \]
for all $W_{\pi} \in W\left(\pi, \psi_{n}\right)$, where $\rho$ denotes the right translation action and $\widetilde{W}_{\pi}(g) = W_{\pi}\left(w_{n}\leftidx^{t}g^{-1}\right)$ for $g \in \GL(n, \E)$.
\end{thm}
\begin{remark}
    \normalfont We recall from \cite{Gelfand, Rodier} that all supercuspidal representations of $\GL(n, \E)$ are generic.
\end{remark}

We end this subsection with the following result of Adrian and Liu that gives the twisted gamma factor between a simple supercuspidal representation $\pi_{(v,  \, \varphi,  \, \zeta)}$ and a tamely ramified quasi-character $\lambda$ of $\E^{\times}$, i.e., $\lambda$ is trivial on $1 + \mathcal{P}_{\E}$ \cite[Corollary 3.12]{Adrian-Liu}:
\begin{proposition}\label{TamelyRamified}
Let $\pi_{(v,  \, \varphi,  \, \zeta)}$ be a simple supercuspidal representation of $\GL(n, \E)$ and $\lambda$ a tamely ramified quasi-character of $\E^{\times}$.~Then 
\[
\gamma\left(s, \, \pi_{(v,  \, \varphi,  \, \zeta)} \times \lambda, \, \psi_{\E} \right) = \zeta^{-1} \, \lambda(-1)^{n-1}\lambda(v \varpi_{\E}) \, q_{\E}^{1/2-s}.
\]
\end{proposition}

\subsection{Distinction by $\GL(n, \F)$}\label{Distinction} In this subsection, we provide necessary conditions for a irreducible smooth complex representation $\tau$ of $\GL(n, \E)$ to be distinguished by $\GL(n, \F)$.

We say that $\tau$ is \textit{distinguished} by $\GL(n, \F)$ if the space
\[
\Hom_{\GL(n, \, \F)}\left(\tau, \mathbbm{1}\right)
\]
is non-zero, with $\mathbbm{1}$ denoting the trivial quasi-character of $\GL(n, \F)$.~Let $\sigma$ be the non-trivial automorphism of the Galois group $\text{Gal}\left(\E/\F\right)$.~Then $\tau$ is \textit{$\sigma$-self-dual} if 
 \begin{equation}\label{self-dual}
        \tau^{\vee} \cong \tau^{\sigma}
\end{equation}
where $\tau^{\vee}$ denotes the contragredient representation of $\tau$ and $\tau^{\sigma} = \tau \circ \sigma$.~If $\tau$ is distinguished by $\GL(n, \F)$, then $\tau$ satisfies the following \cite{Flicker, Hakim-Offen, Prasad1, Prasad2, Secherre}:
\begin{enumerate}
    \item $\tau$ is unitary,
    \item $\tau$ is $\sigma$-self-dual, and
    \item the central character of $\tau$ must be trivial on $\F^{\times}$.
\end{enumerate}

We end this subsection with the following result that will aid us in our computations \cite[Subsection 3.5]{Bushnell-Henniart}:
\begin{theorem}[Duality theorem]\label{Duality}
    \normalfont
    Let $H$ be a closed subgroup of $G$ and $\mu_{H \backslash G}$ a left-invariant Haar measure on $H \backslash G$.~Let $(\rho, W)$ be a smooth representation of $H$.~There is a natural isomorphism
    \begin{equation*}
        \left(\ind_{H}^{G}\rho \right)^{\vee} \cong \Ind_{H}^{G} \left(\delta_{H \backslash G} \otimes \rho^{\vee}\right),
    \end{equation*}
    where 
    \begin{equation*}
        \delta_{H \backslash G} = \delta_{H}^{-1}\delta_{G}\big|_{H} \, \colon \, H \to \mathbb{R}_{>0}^{\times}
    \end{equation*} 
    with $\delta_{G}$ and $\delta_{H}$ denoting the modulus quasi-characters of $G$ and $H$ respectively.
\end{theorem}

\section{Proof of Main Results}\label{Proof}

In this section, we prove Theorem \ref{TheoremMain}.~To do so, we first introduce two supporting lemmas.
\begin{lemma}\label{Lemma1}
    Let $\pi_{(v,  \, \varphi,  \, \zeta)}$ be a simple supercuspidal representation of $\GL(n, \E)$ such that $\omega_{(v,  \, \varphi,  \, \zeta)}$ is trivial on $\F^{\times}$.~Suppose 
    \[
    \gamma\left(\sfrac{1}{2}, \, \pi_{(v,  \, \varphi,  \, \zeta)} \times \lambda, \, \psi_{\E} \right) =  1
    \]
    for all unitary, tamely ramified quasi-characters $\lambda$ of  $\E^{\times}$ that are trivial on $\F^{\times}$.~Then $\overline{v} \in k_{\F}^{\times}$, 
    $\varphi$ is trivial on $k_{\F}^{\times}$, and $\zeta = 1$.
\end{lemma}
\begin{proof}
    Assume that all twisted gamma factors in the statement of the lemma are equal to 1.~From Proposition \ref{TamelyRamified}, we have
    \begin{equation}\label{Gamma}
            \gamma\left(\sfrac{1}{2}, \, \pi_{(v,  \, \varphi,  \, \zeta)} \times \lambda, \, \psi_{\E} \right) = \zeta^{-1}\lambda(v \varpi_{\E}) = 1.
    \end{equation}
    Setting $\lambda = \mathbbm{1}_{\E^{\times}}$, the trivial quasi-character of $\E^{\times}$, we see that $\zeta = 1$.~This fact along with \eqref{Gamma} imply that $\lambda(v \varpi_{\E}) = 1$ for all unitary, tamely ramified quasi-characters $\lambda$ of $\E^{\times}$ that are trivial on $\F^{\times}$.
    
   Next, we note that such $\lambda$ may be regarded as quasi-characters of the finite abelian group 
    \[
    G := \bigslant{\E^{\times}}{\F^{\times}\left(1 + \mathcal{P}_{\E}\right)}.
    \]
    From the Pontryagin duality theorem \cite[p. 53]{Morris}, we have that $G$ is isomorphic to its double dual.~Because the element $v \varpi_{\E}$ lies in the intersection of the kernels of all $\lambda$, this implies $v \varpi_{\E} \in \F^{\times}\left(1 + \mathcal{P}_{\E}\right)$.
    
    Since $\E/\F$ is unramified, we may take $\varpi_{\E} = \varpi_{\F}$.~This further implies that $v \in \mathcal{O}_{\F}^{\times}(1 + \mathcal{P}_{\E})$.~Therefore, Lemma \ref{Lemma1} follows as $\omega_{(v,\,\varphi,\,\zeta)}$ is trivial on $\F^{\times}$.
\end{proof}

\begin{lemma}\label{Lemma2}
    Let $\pi_{(v,  \, \varphi,  \, \zeta)}$ be a simple supercuspidal representation of $\GL(n, \E)$.~Then $\pi_{(v,  \, \varphi,  \, \zeta)}$ is $\sigma$-self-dual if and only if $\overline{v} \in k_{\F}^{\times}$, $\varphi$ is trivial on $k_{\F}^{\times}$ and $\zeta = \pm 1$.
\end{lemma}
\begin{proof}
    We first suppose that $n$ is odd.~From Theorem \ref{Duality}, we have
    \[
    \pi_{(v,  \, \varphi,  \, \zeta)}^{\vee} \cong \ind_{\textbf{J}_{v}}^{\GL(n, \E)}\Lambda_{(v,  \, \varphi,  \, \zeta)}^{-1} \,.
    \]
    This implies that the simple stratum corresponding to $\pi_{(v,  \, \varphi,  \, \zeta)}^{\vee}$ is of the form $\left[\mathfrak{I}_{n}, 1, 0, -\beta_{v} \right]$. 
    
    Moreover, we note that $-\beta_{v}$ is conjugate to $\beta_{-v}$ by the diagonal matrix $g$ defined by: 
    \begin{equation}\label{g}
        g = \begin{pmatrix}
            -1 & & & & \\
            & 1 & & & \\
            & & -1 & & \\
            & & & \ddots & \\
            & & & & (-1)^{n}
        \end{pmatrix}
    \end{equation}
    and
    \[
        \Lambda_{(v,\,\varphi,\,\zeta)}^{-1}(-\beta_v)
        =   \varphi(-1) \, \zeta^{-1}.
    \]
    Using the results in \cite[Chapter 6]{Bushnell-Kutzko} and Subsection \ref{SimpleSupercuspidalRepresentations}, we have
    \begin{equation}\label{Contra}
        \pi_{(v,  \, \varphi,  \, \zeta)}^{\vee} \cong \pi_{\left(-v,  \, \varphi^{-1},  \, \varphi(-1) \, \zeta^{-1} \right)}.
    \end{equation}
    
    By definition, we have that $\varphi^\sigma:=\varphi\circ\sigma$.~Since $\psi_{\E}\circ\sigma=\psi_{\E}^{-1}$, the simple character occurring in $\left(\Lambda_{(v,\,\varphi,\,\zeta)}\circ\sigma\right)$ is associated with $-\beta_{\sigma(v)}$.~Conjugating again by $g$, we obtain
\begin{equation}\label{Galois}
    \pi_{(v,\,\varphi,\,\zeta)}^{\sigma}
    \cong
    \pi_{\left(
        -\sigma(v),\,
        \varphi^\sigma,\,
        \varphi(-1) \, \zeta
    \right)}.
\end{equation}
    From \eqref{self-dual}, (\ref{Contra}), and (\ref{Galois}),
    \[
    \pi_{\left(-\sigma(v),  \, \varphi^{\sigma},  \, \varphi(-1) \, \zeta \right)} \cong \pi_{\left(-v,  \, \varphi^{-1},  \, \varphi(-1) \, \zeta^{-1} \right)}.
    \]
    Using Proposition \ref{Bijection}, we have $-\sigma(v) \equiv -v \pmod{\mathcal{P}_{\E}}$ which implies $\overline{v} \in k_{\F}^{\times}$. 

    If $\varphi^{-1} = \varphi^{\sigma}$, then $\varphi\left(N_{\E/\F}(x)\right) = 1$ for all $x \in \mathcal{O}_{\E}^{\times}$.~Such $\varphi$ are exactly those quasi-characters of $k_{\E}^{\times}$ which are trivial on $k_{\F}^{\times}$.~Moreover, since $\zeta = \zeta^{-1}$, we have that $\zeta = \pm 1$.

    Next, suppose that $n$ is even.~Using similar logic as above, we see that $-\beta_{v}$ is conjugate to $\beta_{v}$ by $g$ given in \eqref{g} and 
    \[
        \pi_{(v,  \, \varphi,  \, \zeta)}^{\vee} \cong \pi_{\left(v,  \, \varphi^{-1},  \, \varphi(-1) \, \zeta^{-1} \right)}.
    \]
    From \eqref{self-dual}, we have that
    \[
        \pi_{\left(\sigma(v),  \, \varphi^{\sigma},  \, \varphi(-1) \, \zeta \right)} \cong \pi_{\left(v,  \, \varphi^{-1},  \, \varphi(-1) \, \zeta^{-1} \right)}.
    \]
    It follows from Proposition \ref{Bijection} that $\sigma(v) \equiv v \pmod{\mathcal{P}_{\E}}$, $\varphi^{-1} = \varphi^{\sigma}$, and $\zeta = \zeta^{-1}$.
    
    If $\sigma(v) \equiv v \pmod{\mathcal{P}_{\E}}$, then $\overline{v} \in k_{\F}^{\times}$.~Finally, the same conditions on $\varphi$ and $\zeta$ hold as in the odd case.

    Conversely, suppose that $\overline{v}\in k_{\F}^{\times}$, $\varphi \big|_{k_{\F}^{\times}}=1$, and $\zeta=\pm1$.~Then $\overline{\sigma(v)}=\overline v$, $\varphi^{\sigma}=\varphi^{-1}$, and $\zeta=\zeta^{-1}$.~The preceding descriptions of $\pi_{(v,\,\varphi,\,\zeta)}^{\sigma}$ and $\pi_{(v,\,\varphi,\,\zeta)}^{\vee}$, together with Proposition \ref{Bijection}, therefore imply that
\[
\pi_{(v,\,\varphi,\,\zeta)}^{\sigma}
\cong
\pi_{(v,\,\varphi,\,\zeta)}^{\vee}. \qedhere
\]
\end{proof}

The preceding two lemmas now enable us to prove Theorem \ref{TheoremMain}.
\begin{proof}[Proof of Theorem \ref{TheoremMain}]
    If $\pi_{(v,  \, \varphi,  \, \zeta)}$ is distinguished by $\GL(n, \F)$, then (2) follows from Theorem \ref{n-2} and \cite[Theorem 4.1]{Hakim}.~Assuming (2), we have that (3) follows from Lemma \ref{Lemma1}.~Finally, we show that (3) implies (1).
    
    From Lemma \ref{Lemma2}, we have that $\pi_{(v,  \, \varphi,  \, \zeta)}$ is $\sigma$-self-dual.~Using Proposition \ref{Bijection}, we may assume that $v \in \mathcal{O}_{\F}^{\times}$.~Moreover, \cite[Corollary 9.6]{Secherre} tells us that any distinguished representation of $\textbf{J}_{v} = \E\left[\beta_{v}\right]^{\times} \J^{1}$ extending $\eta := \psi_{\beta_{v}}$ is $\sigma$-self-dual.~Adopting the notation in \cite[Subsection 5.2]{Secherre}, we have that $\boldsymbol\lambda := \Lambda_{(v,  \, \varphi,  \, \zeta)}$ may be decomposed as
\[
\boldsymbol\lambda = \boldsymbol\kappa \otimes \boldsymbol\rho
\]
where $\boldsymbol\kappa$ is of the following form:~first, we uniquely write an element $x \in \textbf{J}_{v}$ as $x = \beta_{v}^{k} y z$ with $k \in \mathbb{Z}$, $y \in \mu_{\E}'$, and $z \in \J^{1}$.~Next, we set
\[
\boldsymbol\kappa(x) = \psi_{\beta_{v}}(z).
\]
From $\boldsymbol\kappa$, we have that $\boldsymbol\rho$ is given by
\[
\boldsymbol\rho(x) = \zeta^{k}\varphi(y) = \varphi(y).
\]

We claim that $\boldsymbol\kappa$ is distinguished in the sense of \cite{Secherre}.~Once it is distinguished, we have from \cite[Corollary 9.6]{Secherre} that it is also $\sigma$-self-dual.~For elements $x \in \textbf{J}_{v} \cap \GL(n, \F)$, 
\[
\boldsymbol\kappa(x) = \psi_{\beta_{v}}(z)
\]
where $z = (z_{i, j}) \in \J^{1}$ with $z_{i, j} \in \mathcal{O}_{\F}$.~ Hence, 
\[
\psi_{\beta_{v}}(z) = \psi_{\E}\left(\sum_{i=1}^{n-1}z_{i, i+1} + z_{n, 1}v^{-1}\right) = 1
\]
as $\psi_{\E}$ is trivial on $\F$; this proves our claim.

The quasi-character $\boldsymbol\rho$ arises from a certain tamely ramified character $\xi$ of an extension field $\K$ of $\E[\beta_v]$.~Furthermore, the pair $\left(\K/\E[\beta_v], \, \xi \right)$ is called an \textit{admissible pair of level zero} (see \cite[Definition 5.6]{Secherre} for a definition).~In our situation, we have that $\K :=  \E[\beta_v]$ and $\xi$ is a quasi-character of $\E[\beta_v]^\times$ given by
\begin{equation}\label{xiK}
\xi(\beta_v^k yz)=\zeta^k\varphi(y)=\varphi(y),
\end{equation}
with $y\in\mu_{\E}'$ and $z\in 1+\mathcal P_{\E}$.

Let $\D_{0} := F\left[\beta_{v}\right]$.~From \cite[Theorem 10.9]{Secherre}, we have that $\pi_{(v,  \, \varphi,  \, \zeta)}$ is distinguished by $\GL(n, \F)$ if and only if $\xi\big|_{\D_{0}^{\times}}$ is trivial ($\xi\big|_{\D_{0}^{\times}}$ is the quasi-character $\delta_{0}$ of $\D_{0}^{\times}$ in \cite[Theorem 10.9]{Secherre}).~Since $v \in \mathcal{O}_{\F}^{\times}$, it follows that $\D_{0}$ is a degree $n$ totally ramified field extension of $\F$.~Hence, restricting $\xi$ to $\D_{0}^{\times}$ amounts to taking $y \in \mu_{\F}'$ in (\ref{xiK}).~This implies $\xi\big|_{\D_{0}^{\times}}$ is trivial as $\varphi$ is trivial on $\mathcal{O}_{\F}^{\times}$.
\end{proof}


\appendix
\section{}\label{Appendix}

For a simple supercuspidal representation $\pi_{(v,  \, \varphi,  \, \zeta)} $ of $\GL(n, \E)$, we define an explicit Whittaker function $\mathcal{W}_{(v,  \, \varphi,  \, \zeta)} \in W\left(\pi_{(v,  \, \varphi,  \, \zeta)}, \psi_{\E}\right)$ via the following \cite[Section 3.3]{Adrian-Liu}, \cite[Example 2.23]{Ye}:

\hspace{0pt}\resizebox{1.0\linewidth}{!}{
  \begin{minipage}{\linewidth}
\begin{align*}
    g \mapsto  \begin{cases} 
      \hfill \psi_{n}(u)  \Lambda_{\left(v, \, \varphi, \, \zeta\right)}(h) &, \text{ if } g = uh \in \N(n, \E) \, \textbf{J}_{v} \text{ with } u \in \N(n, \E), \, h \in \textbf{J}_{v} \\
      \hfill 0 & ,  \text{ otherwise}
   \end{cases}.
\end{align*}
\end{minipage}
}
Next, we give a useful property about the support $\Su\left(\mathcal{W}_{\left(v, \, \varphi, \, \zeta\right)}\right)$ of $\mathcal{W}_{(v,  \, \varphi,  \, \zeta)}$ on $\text{P}(n, \E)$ that we will refer to in our calculations \cite[Theorem 5.8]{Paskunas}:
\begin{theorem}\label{Support}
\normalfont
    The support of $\mathcal{W}_{\left(v, \, \varphi, \, \zeta\right)}$ satisfies:
    \[
    \Su\left(\mathcal{W}_{(v,  \, \varphi,  \, \zeta)}\right) \cap \text{P}(n, \E) = \N(n, \E)\left(U^{1}\left(\mathfrak{I}_{n}\right) \cap \text{P}(n, \E)\right)
    \]
    and 
    \[
        \mathcal{W}_{(v,  \, \varphi,  \, \zeta)}(um) = \psi_{n}(u)\psi_{\beta_{v}}(m)
    \]
    for all $u \in \N(n, \E)$ and $ m \in U^{1}\left(\mathfrak{I}_{n}\right) \cap \text{P}(n, \E)$.
\end{theorem}

Let $\lambda$ be a depth-one quasi-character of $\E^{\times}$, then $\lambda$ is trivial on $1 + \mathcal{P}_{\E}^{2}$, but non-trivial on $1 + \mathcal{P}_{\E}$.~Hence, there exists $c_{\lambda} \in \mathcal{O}_{\E}^{\times}$ such that for $x \in \mathcal{O}_{\E}$,
\begin{equation}\label{Depth-One}
    \lambda(1 + \varpi_{\E}x) = \psi_{\E}(c_{\lambda}x).
\end{equation}
We first compute the twisted gamma factor between $\pi_{(v,  \, \varphi,  \, \zeta)} $ a simple supercuspidal representation of $\GL(2, \E)$ and $\lambda$.~Let $\alpha = \begin{pmatrix}
    1 & -c_{\lambda}\varpi_{\E}^{-1} \\ 
    & 1
\end{pmatrix}$.
\begin{proposition}\label{PropA1}
    Let $\pi_{(v,  \, \varphi,  \, \zeta)}$ be a simple supercuspidal representation of $\GL(2, \E)$ and $\lambda$ a depth-one quasi-character of $\E^{\times}$.~Then
    \[
    \Psi\left(s; \, \rho(\alpha)\mathcal{W}_{\left(v, \, \varphi, \, \zeta\right)}, \, \lambda\right) = \psi_{\E}\left(-c_{\lambda}\varpi_{\E}^{-1}\right) .
    \]
\end{proposition}
\begin{proof}
From Theorem \ref{Support}, we have 
\begin{align*}
\Psi\left(s; \, \rho(\alpha)\mathcal{W}_{\left(v, \, \varphi, \, \zeta\right)}, \, \lambda \right) &= 
   \int_{1+\mathcal{P}_{\E}}\mathcal{W}_{\left(v, \, \varphi, \, \zeta\right)}\left(\begin{pmatrix}
    h &  \\
     & 1
\end{pmatrix}\begin{pmatrix}
    1 & -c_{\lambda}\varpi_{\E}^{-1} \\
     & 1
\end{pmatrix}\right)\lambda(h) \, dh \\
& = \int_{1+\mathcal{P}_{\E}}\mathcal{W}_{\left(v, \, \varphi, \, \zeta\right)}\left(\begin{pmatrix}
    1 & -hc_{\lambda}\varpi_{\E}^{-1} \\
     & 1
\end{pmatrix}\begin{pmatrix}
    h &  \\
     & 1
\end{pmatrix}\right)\lambda(h) \, dh \\
&= \int_{1+\mathcal{P}_{\E}}\psi_{\E}\left(-hc_{\lambda}\varpi_{\E}^{-1}\right)\lambda(h) \, dh \\ 
&= \psi_{\E}\left(-c_{\lambda}\varpi_{\E}^{-1}\right)\int_{\mathcal{O}_{\E}}\psi_{\E}\left(-c_{\lambda}x\right)\lambda(1 + \varpi_{\E}x) \, dx \\
&= \psi_{\E}\left(-c_{\lambda}\varpi_{\E}^{-1}\right)
\end{align*}
as we can rewrite $h = 1 + \varpi_{\E}x$ where $x \in \mathcal{O}_{\E}$.
\end{proof}

Next, we compute $\widetilde{\Psi}\left(1-s; \, \rho(w_{2, 1})\widetilde{\rho(\alpha)\mathcal{W}_{\left(v, \, \varphi, \, \zeta\right)}}, \, \lambda^{-1}\right)$.
\begin{proposition}\label{PropA2}
Let $\pi_{(v,  \, \varphi,  \, \zeta)}$ be a simple supercuspidal representation of $\GL(2, \E)$ and $\lambda$ a depth-one quasi-character of $\E^{\times}$.~Then
\begin{align*}
    & \widetilde{\Psi}\left(1-s; \, \rho(w_{2, 1})\widetilde{\rho(\alpha)\mathcal{W}_{\left(v, \, \varphi, \, \zeta\right)}}, \, \lambda^{-1}\right) \\
    &= \omega_{(v,  \, \varphi,  \, \zeta)}\left(c_{\lambda}^{-1}\varpi_{\E}\right)\lambda\left(-c_{\lambda}^{-2}\varpi_{\E}^{2}\left(1 - (c_{\lambda}^{2}v)^{-1}\varpi_{\E}\right)\right)\psi_{\E}\left(c_{\lambda}\varpi_{\E}^{-1}\right)q_{\E}^{1-2s}.
\end{align*}
\end{proposition}
\begin{proof}
From Subsection \ref{TwistedLocalFactors}, we have by definition that 
\begin{align*}
    & \widetilde{\Psi}\left(1-s; \, \rho(w_{2, 1})\widetilde{\rho(\alpha)\mathcal{W}_{\left(v, \, \varphi, \, \zeta\right)}}, \, \lambda^{-1}\right) \\
    & = \int_{\E^{\times}}\mathcal{W}_{\left(v, \, \varphi, \, \zeta\right)} \left(\begin{pmatrix}
        & 1 \\
        h^{-1} & -h^{-1}c_{\lambda}\varpi_{\E}^{-1}
    \end{pmatrix}\right)\lambda(h)^{-1}|h|^{\sfrac{1}{2}-s}\, dh.
\end{align*}
We first prove the following lemma that tells us when matrices of the form 
\[
\begin{pmatrix}
        & 1 \\
        h^{-1} & -h^{-1}c_{\lambda}\varpi_{\E}^{-1}
    \end{pmatrix}
\]
belong to $\Su\left(\mathcal{W}_{\left(v, \, \varphi, \, \zeta\right)}\right)$, which is contained in the disjoint union of double cosets of the form 
\[
\N(2, \E)  \beta_{v}^{k}  \mathcal{O}_{\E}^{\times}  U^{1}(\mathfrak{I}_{2})
\]
with $k \in \mathbb{Z}$.
    
\begin{lemma}\label{LemmaA}
    Let $h \in \E^{\times}$, then 
    \[
    \begin{pmatrix}
        & 1 \\
        h^{-1} & -h^{-1}c_{\lambda}\varpi_{\E}^{-1}
    \end{pmatrix} \in \left\{ \, \N(2, \E)  \beta_{v}^{k}  \mathcal{O}_{\E}^{\times}  U^{1}(\mathfrak{I}_{2}) \, : \, k \in \mathbb{Z} \, \right\}
    \]
    if and only if $k = -2$.
\end{lemma}
\begin{proof}
     Let $M(h) := \begin{pmatrix}
        & 1 \\
        h^{-1} & -h^{-1}c_{\lambda}\varpi_{\E}^{-1}
    \end{pmatrix}$, then $\text{val}_{\E}\left(M(h)_{(2, \, 1)}\right) > \text{val}_{\E}\left(M(h)_{(2, \, 2)}\right)$.~This implies that $k$ cannot be odd, as this would imply that $\text{val}_{\E}\left(M(h)_{(2, \, 1)}\right) \leq  \text{val}_{\E}\left(M(h)_{(2, \, 2)}\right)$. Consequently, the power of $\beta_{v}$ must be even.

    Suppose $M(h) \in \Su\left(\mathcal{W}_{\left(v, \, \varphi, \, \zeta\right)}\right)$.~Then $M(h) = u\beta_{v}^{k}z$ where $u \in \N(2, \E)$, $k$ is even, and $z \in \mathcal{O}_{\E}^{\times}U^{1}\left(\mathfrak{I}_{2}\right)$.~Comparing the $(1, 1)$-entries of $M(h)$ and $u\beta_{v}^{k}z$, we see that $\text{val}_{\E}\left(u_{1, 2}\right) = -1$.~Rewriting $k = 2j$, we have from comparing the $(1, 2)$-entries of $M(h)$ and $u\beta_{v}^{k}z$ that $\text{val}_{\E}\left(u_{1, 2}\right) = j$ if $j < 0$ and $\text{val}_{\E}\left(u_{1, 2}\right) \geq 0$ if $j \geq 0$.~Hence, $j = -1$.

    The preceding paragraph implies that $M(h) \in \N(2, \E)  \beta_{v}^{-2}  \mathcal{O}_{\E}^{\times}  U^{1}(\mathfrak{I}_{2})$. Let
    \begin{align}\label{Decomposition}
        M(h) &= u\beta_{v}^{-2}z \\ \notag
        &= \begin{pmatrix}
            1 & u_{1, 2} \\
            & 1
        \end{pmatrix}(av\varpi_{\E})\begin{pmatrix}
            1 + z_{1, 1}\varpi_{\E} & z_{1, 2} \\
            z_{2, 1}\varpi_{\E} & 1 + z_{2, 2}\varpi_{\E}
        \end{pmatrix}
    \end{align}
    where $a \in \mathcal{O}_{\E}^{\times}$, $u_{1, 2} \in \E$, and $z_{i, j} \in \mathcal{O}_{\E}$.~From (\ref{Decomposition}), $z_{2, 1} = -c_{\lambda}^{-1}(1 + z_{2, 2}\varpi_{\E})$.~Comparing the $(1, 1)$-entries of $M(h)$ and $u\beta_{v}^{-2}z$, we solve for $u_{1, 2}$ and see that $u_{1, 2} \in c_{\lambda}\varpi_{\E}^{-1}\left(1 + \mathcal{P}_{\E}\right)$.~Using this, we compare the $(1, 2)$-entries of $M(h)$ and $u\beta_{v}^{-2}z$ to see that $a \equiv (vc_{\lambda})^{-1} \pmod{\mathcal{P}_{\E}}$.

    Our computations imply that $h \in -c_{\lambda}^{2}\varpi_{\E}^{-2}(1 + \mathcal{P}_{\E})$ and we may factorize $M(h)$ as:
    \begin{equation}\label{Factorization}
        M(h) = \begin{pmatrix}
            1 & -hc_{\lambda}^{-1}\varpi_{\E} \\
            & 1
        \end{pmatrix}\beta_{v}^{-2}(vc_{\lambda})^{-1}\begin{pmatrix}
            1 & \\
            h^{-1}c_{\lambda}\varpi_{\E}^{-1} & h^{-1}c_{\lambda}^{2}\varpi_{\E}^{-2}
        \end{pmatrix}. \qedhere
    \end{equation} 
\end{proof}

Lemma \ref{LemmaA} tells us that $h \in -c_{\lambda}^{2}\varpi_{\E}^{-2}(1 + \mathcal{P}_{\E})$.~Performing a change of variables, we let $h' = -c_{\lambda}^{-2}\varpi_{\E}^{2}h$.~Using the decomposition at the end of Lemma \ref{LemmaA}, we have that 
\begin{align*}
    & \widetilde{\Psi}\left(1-s; \, \rho(w_{2, 1})\widetilde{\rho(\alpha)\mathcal{W}_{\left(v, \, \varphi, \, \zeta\right)}}, \, \lambda^{-1}\right) \\
    &= \zeta^{-2}\varphi\left((c_{\lambda}v)^{-1}\right)\lambda\left(-c_{\lambda}^{-2}\varpi_{\E}^{2}\right)q_{\E}^{1-2s}\int_{1 + \mathcal{P}_{\E}}\psi_{\E}\left(h'c_{\lambda}\varpi_{\E}^{-1} - (h'c_{\lambda}v)^{-1}\right)\lambda(h')^{-1} \, dh' \\
    &= \omega_{(v,  \, \varphi,  \, \zeta)}\left(c_{\lambda}^{-1}\varpi_{\E}\right)\lambda\left(-c_{\lambda}^{-2}\varpi_{\E}^{2}\right)\psi_{\E}\left(-(c_{\lambda}v)^{-1}\right)\psi_{\E}\left(c_{\lambda}\varpi_{\E}^{-1}\right)q_{\E}^{1-2s} \\ & \hspace{3in} \cdot \int_{\mathcal{O}_{\E}}\psi_{\E}\left(c_{\lambda}x\right)\lambda(1 + \varpi_{\E}x)^{-1} \, dx \\
    &= \omega_{(v,  \, \varphi,  \, \zeta)}\left(c_{\lambda}^{-1}\varpi_{\E}\right)\lambda\left(-c_{\lambda}^{-2}\varpi_{\E}^{2}\left(1 - (c_{\lambda}^{2}v)^{-1}\varpi_{\E}\right)\right)\psi_{\E}\left(c_{\lambda}\varpi_{\E}^{-1}\right)q_{\E}^{1-2s}
\end{align*}
as $\psi_{\E}\left(-(c_{\lambda}v)^{-1}\right) = \lambda\left(1 - (c_{\lambda}^{2}v)^{-1}\varpi_{\E}\right)$ and $h' = 1 + \varpi_{\E}x$ for $x \in \mathcal{O}_{\E}$.
\end{proof}

Putting our computations together, we have the following formula:
\begin{corollary}\label{TD1}
Let $\pi_{(v,  \, \varphi,  \, \zeta)}$ be a simple supercuspidal representation of $\GL(2, \E)$ and $\lambda$ a depth-one quasi-character of $\E^{\times}$.~Then
\[
\gamma\left(s, \pi_{(v,  \, \varphi,  \, \zeta)} \times \lambda, \psi_{\E}\right) = \omega_{(v,  \, \varphi,  \, \zeta)}\left(c_{\lambda}^{-1}\varpi_{\E}\right)\lambda\left(c_{\lambda}^{-2}\varpi_{\E}^{2}\left(1 - (c_{\lambda}^{2}v)^{-1}\varpi_{\E}\right)\right)\psi_{\E}\left(2c_{\lambda}\varpi_{\E}^{-1}\right)q_{\E}^{1-2s}.
\]
\end{corollary}
\begin{proof}
Using Theorem \ref{GammaFactor2.7} and Propositions \ref{PropA1} and \ref{PropA2}, we have
    \begin{align*}
    & \gamma\left(s, \pi_{(v,  \, \varphi,  \, \zeta)} \times \lambda, \psi_{\E}\right) \\
    &= \frac{\omega_{(v,  \, \varphi,  \, \zeta)}\left(c_{\lambda}^{-1}\varpi_{\E}\right)\lambda\left(c_{\lambda}^{-2}\varpi_{\E}^{2}\left(1 - (c_{\lambda}^{2}v)^{-1}\varpi_{\E}\right)\right)\psi_{\E}\left(c_{\lambda}\varpi_{\E}^{-1}\right)q_{\E}^{1-2s}}{\psi_{\E}\left(-c_{\lambda}\varpi_{\E}^{-1}\right)} \\
    &= \omega_{(v,  \, \varphi,  \, \zeta)}\left(c_{\lambda}^{-1}\varpi_{\E}\right)\lambda\left(c_{\lambda}^{-2}\varpi_{\E}^{2}\left(1 - (c_{\lambda}^{2}v)^{-1}\varpi_{\E}\right)\right)\psi_{\E}\left(2c_{\lambda}\varpi_{\E}^{-1}\right)q_{\E}^{1-2s}.~\qedhere
\end{align*}
\end{proof}

If $\pi_{(v,\,\varphi,\,\zeta)}$ is distinguished by $\GL(2,\F)$, then the twisted gamma factors between $\pi_{(v,\,\varphi,\,\zeta)}$ and every tamely ramified quasi-character or depth-one quasi-character of $\E^{\times}$ that is unitary and trivial on $\F^{\times}$ are equal to 1.~From Lemma \ref{Lemma1}, we have that $\overline{v} \in k_{\F}^{\times}$, $\varphi$ is trivial on $k_{\F}^{\times}$, and $\zeta = 1$.

For depth-one quasi-characters, we have that for $x \in \mathcal{O}_{\E}$,
\begin{equation}\label{Level-OneE}
    \lambda(1+\varpi_{\F}x) = \psi_{\E}\left(c_{\lambda}x\right)
\end{equation}
where $c_{\lambda} \in \mathcal{O}_{\E}^{\times}$.~This implies that for $x \in \mathcal{O}_{\F}$, the right hand side of (\ref{Level-OneE}) is equal to 1.~Exploiting this observation, we have the following lemma:
\begin{lemma}\label{LemmaA2}
The reduction $\overline{c_{\lambda}}$ lies in $k_{\F}^{\times}$.
\end{lemma}
\begin{proof}
    Let $\sigma_{f}$ be the root of some monic degree two polynomial $f \in  \mathcal{O}_{\F}[X]$ such that the reduction of the coefficients of $f$ modulo $\mathcal{P}_{\F}$ is irreducible over $k_{\F}$.~Then $\E = \F\left[\sigma_{f}\right]$ and $c_{\lambda} = c_{0} + c_{1}\sigma_{f}$ with $c_{i} \in \mathcal{O}_{\F}$.~For $x \in \mathcal{O}_{\F}$, we have
\[
\psi_{\E}(c_{0}x + c_{1}x\sigma_{f}) = \psi_{\E}(c_{1}x\sigma_{f}) = 1.
\]
Hence, $c_{1} \in \mathcal{P}_{\F}$; otherwise, $\psi_{\E}$ would be trivial on $\mathcal{O}_{\E}$.
\end{proof}

Lemma \ref{LemmaA2} enables us to prove the following:
\begin{proposition}
    Let $\overline{v} \in k_{\F}^{\times}$, $\varphi$ be trivial on $k_{\F}^{\times}$, and $\zeta = 1$.~Then 
    \[
    \gamma\left(\sfrac{1}{2}, \, \pi_{(v,  \, \varphi,  \, \zeta)} \times \lambda, \, \psi_{\E}\right) = 1
    \]
    for all unitary, depth-one quasi-characters $\lambda$ of $\E^{\times}$ that are trivial on $\F^{\times}$.
\end{proposition}
\begin{proof}
    Let $c_{\lambda} = c_{0}(1 + c_{1}\varpi_{\F})$ for some $c_{0} \in \mathcal{O}_{\F}^{\times}$ and $c_{1} \in \mathcal{O}_{\E}$.~By direct computation, we see that 
    \begin{equation}\label{EQ}
        \lambda\left(c_{\lambda}^{-2}\varpi_{\F}^{2}\left(1 - (c_{\lambda}^{2}v)^{-1}\varpi_{\F}\right)\right)\psi_{\E}\left(2c_{\lambda}\varpi_{\F}^{-1}\right) = 1
    \end{equation}
    as $\lambda\left(c_{\lambda}^{-2}\varpi_{\F}^{2}\right) = \lambda(1 + \varpi_{\F}c_{1})^{-2} = \psi_{\E}(-2c_{0}c_{1})$ and $\psi_{\E}\left(2c_{\lambda}\varpi_{\F}^{-1}\right) = \psi_{\E}(2c_{0}c_{1})$.

Since $\overline{v} \in k_{\F}^{\times}$ and $\omega_{(v,  \, \varphi,  \, \zeta)}$ is trivial on $\F^{\times}$, we have from Corollary \ref{TD1} and (\ref{EQ}) that
\begin{align*}
    &\gamma\left(\sfrac{1}{2}, \, \pi_{(v,  \, \varphi,  \, \zeta)} \times \lambda, \, \psi_{\E}\right) \\
    &= \omega_{(v,  \, \varphi,  \, \zeta)}\left(c_{\lambda}^{-1}\varpi_{\F}\right)\lambda\left(c_{\lambda}^{-2}\varpi_{\F}^{2}\left(1 - (c_{\lambda}^{2}v)^{-1}\varpi_{\F}\right)\right)\psi_{\E}\left(2c_{\lambda}\varpi_{\F}^{-1}\right)q_{\E}^{1-2\left(1/2\right)} \\
    &= 1 \qedhere.
\end{align*}
\end{proof}

\newcommand{\etalchar}[1]{$^{#1}$}

\end{document}